\documentclass{amsart}

\usepackage{amssymb}
\usepackage{amsmath}
\usepackage{latexsym}
\usepackage{mathrsfs}
\usepackage{comment}
\usepackage{graphicx}
\usepackage{caption}
\usepackage{color}

\numberwithin{equation}{section}
\newtheorem{thm}[equation]{Theorem}

\newtheorem{rem}[equation]{Remark}
\newtheorem{lem}[equation]{Lemma}

\begin{document}

\title[biharmonic Neumann eigenvalues]
 {A note on the Neumann eigenvalues of the biharmonic operator}

\author{Luigi Provenzano}
\address{Dipartimento di Matematica\\
Universit\`a degli Studi di Padova\\
Via Trieste, 63\\
35126 Padova\\
Italy}
\email{luigiprovenz@gmail.com}




\begin{abstract}
We study the dependence of the eigenvalues of the biharmonic operator subject to Neumann boundary conditions on the Poisson's ratio $\sigma$. In particular, we prove that the Neumann eigenvalues are Lipschitz continuous with respect to $\sigma\in[0,1[$ and that all the Neumann eigenvalues tend to zero as $\sigma\rightarrow 1^-$. Moreover, we show that the Neumann problem defined by setting $\sigma=1$ admits a sequence of positive eigenvalues of finite multiplicity which are not limiting points for the Neumann eigenvalues with $\sigma\in[0,1[$ as $\sigma\rightarrow 1^-$, and which coincide with the Dirichlet eigenvalues of the biharmonic operator.
\end{abstract}

\maketitle

\maketitle

\section{Introduction}
Let $\Omega$ be a bounded domain in $\mathbb R^N$ (i.e., a bounded connected open set) of class $C^{4,\alpha}$ for some $\alpha\in]0,1[$. Let $\sigma\in[0,1[$. We consider the Neumann eigenvalue problem for the biharmonic operator, namely the problem
\begin{equation}\label{Neumann_sigma}
\begin{cases}
\Delta^2 u=\lambda u, & {\rm in}\ \Omega,\\
(1-\sigma)\frac{\partial^2 u}{\partial\nu^2}+\sigma\Delta u=0, & {\rm on}\ \partial\Omega,\\
\frac{\partial\Delta u}{\partial\nu}+(1-\sigma){\rm div}_{\partial\Omega}\left(D^2 u\cdot\nu\right)_{\partial\Omega}=0, & {\rm on}\ \partial\Omega,
\end{cases}
\end{equation}
in the unknowns $u$ (the eigenfunction) and $\lambda$ (the eigenvalue). Here $\nu$ is the outer unit normal to $\partial\Omega$, ${\rm div}_{\partial\Omega}F$ denotes the tangential divergence  of a vector field $F$, which is defined by ${\rm div}_{\partial\Omega}F={\rm div}F_{|_{\partial\Omega}}-(DF\cdot\nu)\cdot\nu$, $F_{\partial\Omega}$ denotes the projection of a vector field $F$ onto the tangent space to $\partial\Omega$, and $D^2u$ is the Hessian matrix of $u$ (we refer to \cite{chasman} for the derivation of the boundary conditions in \eqref{Neumann_sigma}). For $N=2$ this problem is related to the study of the transverse vibrations of a thin plate with a free edge and which occupies at rest a planar region of shape $\Omega$. The coefficient $\sigma$ represents the Poisson's ratio of the material the plate is made of. We refer e.g., to \cite{cohil} for more details on the physical interpretation of problem \eqref{Neumann_sigma} and on the Poisson's ratio $\sigma$. We mention the paper \cite{duffin}, where the author studies the dependence of the vibrational modes  of a plate subject to homogeneous boundary conditions upon the Poisson's ratio $\sigma\in]0,\frac{1}{2}[$, providing also a perturbation formula for the frequencies as functions of the Poisson's coefficient. 

We note that eigenvalue problems for the biharmonic operator have gained significant attention in the last decades. In particular, there are several papers concerning the dependence of the eigenvalues upon different parameters which enter the problem, such as the shape or the coefficients. We refer to the book \cite{henrot} for more information on shape optimization problems for the biharmonic operator. We also refer to \cite{laproeurasian} where it is discussed the dependence of the eigenvalues of polyharmonic operators upon variation of the mass density, and to \cite{buosoprovenzano} where the authors consider Neumann and Steklov-type eigenvalue problems for the biharmonic operator with particular attention to shape optimization and mass concentration phenomena. We also mention \cite{buosopalinuro}, where the author considers the shape sensitivity problem for the eigenvalues of the biharmonic operator (in particular, also those of problem \eqref{Neumann_sigma}) for $\sigma\in]-\frac{1}{N-1},1[$. We note that other issues have been addressed in the literature for polyharmonic operators, such as analyticity, continuity and stability estimates for the eigenvalues with respect to the shape; we refer to \cite{arrietalamberti0,arrietalamberti1,buosolamberti1,buosolamberti2,burenkovlamberti2,burenkovlamberti1} and the references therein.

We recall that problem \eqref{Neumann_sigma} admits an infinite sequence of non-negative eigenvalues of finite multiplicity which depend on $\sigma\in[0,1[$ and which we denote here by
$$
0=\lambda_1(\sigma)=\lambda_2(\sigma)=\cdots=\lambda_{N+1}(\sigma)<\lambda_{N+2}(\sigma)\leq\cdots\leq\lambda_j(\sigma)\leq\cdots.
$$
We note that $\lambda=0$ is an eigenvalue of \eqref{Neumann_sigma} of multiplicity $N+1$, and a set of linearly independent eigenfunctions associated with $\lambda=0$ is given by $\left\{1,x_1,...,x_N\right\}$.

If we set $\sigma=1$, problem \eqref{Neumann_sigma} reads
\begin{equation}\label{Neumann_bad}
\begin{cases}
\Delta^2 u=\lambda u, & {\rm in}\ \Omega,\\
\Delta u=0, & {\rm on}\ \partial\Omega,\\
\frac{\partial\Delta u}{\partial\nu}=0, & {\rm on}\ \partial\Omega.
\end{cases}
\end{equation}
We note that the differential operator associated with problem \eqref{Neumann_bad} is not a Fredholm operator. Indeed all the harmonic functions in $\Omega$ are eigenfunctions corresponding to the eigenvalue $\lambda=0$. We also note that the boundary conditions in \eqref{Neumann_bad} do not satisfy the so-called `complementing conditions' (see \cite[\S 10]{agmon1} and \cite{gazzola} for details), which are necessary conditions for the well-posedness of a differential problem. Nevertheless, problem \eqref{Neumann_bad} admits a countable number of positive eigenvalues of finite multiplicity diverging to $+\infty$, which we denote here by
\begin{equation*}
0<\lambda_1\leq\lambda_2\leq\cdots\leq\lambda_j\leq\cdots.
\end{equation*}

In this paper we show that $\lambda_j(\sigma)\rightarrow 0$ as $\sigma\rightarrow 1^-$ for all $j\in\mathbb N$. Thus, the positive eigenvalues of problem \eqref{Neumann_bad} are not limiting points for the eigenvalues of problem \eqref{Neumann_sigma} as $\sigma\rightarrow 1^-$. Moreover, we show that the positive eigenvalues $\lambda_j$ of problem \eqref{Neumann_bad} coincide with the eigenvalues of the Dirichlet problem for the biharmonic operator, namely the problem
\begin{equation}\label{Dirichlet}
\begin{cases}
\Delta^2 w=\mu w, & {\rm in}\ \Omega,\\
w=0, & {\rm on}\ \partial\Omega,\\
\frac{\partial w}{\partial\nu}=0, & {\rm on}\ \partial\Omega.
\end{cases}
\end{equation}
We recall that, for $N=2$, problem \eqref{Dirichlet} models the transverse vibrations of a thin plate which has a clamped edge (see e.g., \cite{cohil} for details). We also recall that the eigenvalues of \eqref{Dirichlet} are positive and of finite multiplicity and form an increasing sequence diverging to $+\infty$, which we denote here by
\begin{equation}\label{dirichlet_sequence}
0<\mu_1\leq\mu_2\leq\cdots\leq\mu_j\leq\cdots.
\end{equation}

The present paper is organized as follows: in Section \ref{sec:2} we characterize the eigenvalues of problems \eqref{Neumann_sigma}, \eqref{Neumann_bad} and \eqref{Dirichlet}. In Section \ref{sec:3} we prove that all the eigenvalues of problem \eqref{Neumann_sigma} go to zero as $\sigma\rightarrow 1^-$ and moreover, we prove that $\lambda_j=\mu_j$, for all $j\in\mathbb N$. 
Finally, in Section \ref{sec:4}, we consider problems \eqref{Neumann_sigma}, \eqref{Neumann_bad} and \eqref{Dirichlet} in the case of the unit ball in $\mathbb R^N$ centered at zero, where it is possible to recover the results of Section \ref{sec:3} thanks to explicit computations.

\section{Eigenvalues of Neumann and Dirichlet problems}\label{sec:2}
We consider problems \eqref{Neumann_sigma}, \eqref{Neumann_bad} and \eqref{Dirichlet} in their weak formulation. The weak formulation of problem \eqref{Neumann_sigma} when $\sigma\in[0,1[$ is
\begin{equation}\label{weak}
\int_{\Omega}(1-\sigma)D^2u:D^2\varphi+\sigma\Delta u\Delta\varphi dx=\lambda\int_{\Omega}u\varphi dx\,,
\end{equation}
for all $\varphi\in H^2(\Omega)$, in the unknowns $u\in H^2(\Omega)$, $\lambda\in\mathbb R$, where $D^2u:D^2\varphi=\sum_{i,j=1}^N\frac{\partial^2u}{\partial x_i\partial x_j}\frac{\partial^2\varphi}{\partial x_i\partial x_j}$ denotes the Frobenius product. Actually we will recast problem \eqref{weak} in $H^2(\Omega)/\mathcal{N}$, where $\mathcal{N}\subset H^2(\Omega)$ is the subspace of $H^2(\Omega)$ generated by the functions $\lbrace{1,x_1,...,x_N\rbrace}$. To do so, we set
$$
H^2_{\mathcal{N}}(\Omega):=\left\{u\in H^2(\Omega)\,:\ \int_\Omega u dx=\int_{\Omega}\frac{\partial u}{\partial x_i}dx=0\,,\forall i=1,...,N\right\}.
$$
In the sequel we will think of the space $H^2_{\mathcal{N}}(\Omega)$ as endowed with the bilinear form given by the left-hand side of \eqref{weak}. From the fact that $|D^2u|^2\geq\frac{1}{N}(\Delta u)^2$ for all $u\in H^2(\Omega)$ and from the Poincar\'e-Wirtinger inequality, it follows that such bilinear form defines on $H^2_{\mathcal{N}}(\Omega)$ a scalar product whose induced norm is equivalent to the standard one. We denote by $\pi_{\mathcal{N}}$ the map from $H^2(\Omega)$ to $H^2_{\mathcal{N}}(\Omega)$ defined by
$$
\pi_{\mathcal{N}}[u]:=u-\frac{1}{|\Omega|}\int_{\Omega}u+\frac{1}{|\Omega|^2}\sum_{i=1}^N\left(\int_{\Omega}\frac{\partial u}{\partial x_i}dx\right)\left(\int_{\Omega}x_i dx\right)-\frac{1}{|\Omega|}\sum_{i=1}^N\left(\int_{\Omega}\frac{\partial u}{\partial x_i}dx\right)x_i,
$$
for all $u\in H^2(\Omega)$. We denote by $\pi^{\sharp}_{\mathcal{N}}$ the map from $H^2(\Omega)/{\mathcal{N}}$ onto $H^2_{\mathcal{N}}(\Omega)$ defined by the equality $\pi_{\mathcal{N}}=\pi^{\sharp}_{\mathcal{N}}\circ p$, where $p$ is the canonical projection of $H^2(\Omega)$ onto $H^2(\Omega)/N$. The map $\pi^{\sharp}_{\mathcal{N}}$ turns out to be a homeomorphism. Let $F(\Omega)$ be defined by 
$$
F(\Omega):=\left\{G\in H^2(\Omega)':G[1]=G[x_i]=0\,,\ \forall i=1,..,N\right\}.
$$
Then we consider the operator $\mathcal P_{\sigma}$ as an operator from $H^2_{\mathcal{N}}(\Omega)$ to $F(\Omega)$ defined by
\begin{equation*}
\mathcal P_{\sigma}[u][\varphi]:=\int_{\Omega}(1-\sigma)D^2u:D^2\varphi+\sigma\Delta u\Delta\varphi dx\,,\ \ \ \forall u\in H^2_{\mathcal{N}}(\Omega),\varphi\in H^2(\Omega).
\end{equation*}
It turns out that $\mathcal P_{\sigma}$ is a homeomorphism of $H^2_{\mathcal{N}}(\Omega)$ onto $F(\Omega)$. We denote by $\mathcal J$ the continuous embedding of $L^2(\Omega)$  into $H^2(\Omega)'$ defined by
$$
\mathcal J[u][\varphi]:=\int_{\Omega}u\varphi dx\,,\ \ \ \forall u\in L^2(\Omega),\varphi\in H^2(\Omega).
$$
Finally, we define the operator $T_{\sigma}$ acting on $H^2(\Omega)/{\mathcal{N}}$ as follows:
$$T_{\sigma}=(\pi^{\sharp}_{\mathcal{N}})^{(-1)}\circ\mathcal P_{\sigma}^{(-1)}\circ\mathcal J\circ i\circ\pi^{\sharp}_{{\mathcal{N}}},$$
where $i$ denotes the embedding of $H^2(\Omega)$ into $L^2(\Omega)$.
\begin{lem}
The pair $(\lambda,u)$ of the set $(\mathbb R\setminus \lbrace 0\rbrace)\times(H^2_{\mathcal{N}}(\Omega)\setminus\lbrace 0\rbrace)$ satisfies \eqref{weak} if and only if $\lambda>0$ and the pair $(\lambda^{-1},p[u])$ of the set $\mathbb R\times(H^2(\Omega)/{\mathcal{N}}\setminus\lbrace 0\rbrace)$ satisfies the equation $\lambda^{-1}p[u]=T_{\sigma}p[u]$.
\end{lem}
We have the following theorem.
\begin{thm}\label{eigenvalues_sigma}
The operator $T_{\sigma}$ is a non-negative compact self-adjoint operator in $H^2(\Omega)/{\mathcal{N}}$, whose eigenvalues coincide with the reciprocals of the positive eigenvalues of problem \eqref{weak}. In particular, the set of eigenvalues of problem \eqref{weak} is contained in $[0,+\infty[$ and consists of a sequence increasing to $+\infty$ and each eigenvalue has finite multiplicity. Moreover the first eigenvalue is $\lambda=0$ and has multiplicity $N+1$, and a set of linearly independent eigenfunctions corresponding to $\lambda=0$ is given by $\lbrace1,x_1,...,x_N\rbrace$.
\end{thm}
\begin{proof}
It is easy to prove that the operator $T_{\sigma}$ is self-adjoint. The compactness of the operator $T_{\sigma}$ follows from the compactness of the embedding $i$. The last statement is straightforward.
\end{proof}

In an analogous way it is possible to show that the eigenvalues of \eqref{Dirichlet} are positive and of finite multiplicity.  In fact, the weak formulation of problem \eqref{Dirichlet} reads: find $(u,\lambda)\in H^2_0(\Omega)\times\mathbb R$ such that $u$ solves equation $\int_{\Omega}\Delta u\Delta\varphi dx=\lambda\int_{\Omega}u\varphi dx$ for all $\varphi\in H^2_0(\Omega)$. We note that this is equivalent to finding $(u,\lambda)\in H^2_0(\Omega)\times\mathbb R$ such that equation \eqref{weak} holds for all $\varphi\in H^2_0(\Omega)$. From the Poincar\'e inequality it follows that the bilinear form given by the left-hand side of \eqref{weak} defines on $H^2_0(\Omega)$ a scalar product whose induced norm is equivalent to the standard one. Therefore the analogous of Theorem \ref{eigenvalues_sigma} holds, hence the eigenvalues of problem \eqref{Dirichlet} are positive and can be represented by means of an infinite sequence diverging to $+\infty$ of the form \eqref{dirichlet_sequence}, and the corresponding eigenfunctions form a orthonormal basis of $H^2_0(\Omega)$.

Finally, we show that problem \eqref{Neumann_bad} admits an infinite sequence of positive eigenvalues. We have already observed that all harmonic functions in $H^2(\Omega)$ are eigenfunctions corresponding to the eigenvalue $\lambda=0$. We start by recalling the following direct decomposition of the space $H^2(\Omega)$ (see \cite[Theorem 4.7]{borovikov} for details):
\begin{equation*}
H^2(\Omega)=H^2_h(\Omega)\oplus\Delta(H^4(\Omega)\cap H^2_0(\Omega)),
\end{equation*}
where $H^2_h(\Omega)=\left\{h\in H^2(\Omega):\Delta h=0\right\}$ is the space of harmonic functions in $H^2(\Omega)$.\\
In order to characterize the positive eigenvalues of problem \eqref{Neumann_bad} and to get rid of the harmonic functions which are the eigenfunctions associated with $\lambda=0$, we will obtain a problem in $\Delta(H^4(\Omega)\cap H^2_0(\Omega))$. Thus we consider the following weak formulation of problem \eqref{Neumann_bad} for $\lambda\ne 0$.
\begin{equation}\label{bilinear4}
\int_{\Omega}\Delta^2u\Delta^2\varphi dx=\lambda\int_{\Omega}\Delta u\Delta\varphi\,,\ \ \ \forall u,\varphi\in H^4(\Omega)\cap H^2_0(\Omega),
\end{equation}
in the unknowns $u\in H^4(\Omega)\cap H^2_0(\Omega)$, $\lambda\in\mathbb R$  (In the case $\lambda=0$, the solutions of \eqref{Neumann_bad} are exactly the harmonic functions in $H^2(\Omega)$). We note that  there exists a constant $C>0$ such that $\int_{\Omega}\Delta^2u\Delta^2\varphi dx\leq C\|u\|_{H^4(\Omega)}\|\varphi\|_{H^4(\Omega)}$ and $\|u\|_{H^4(\Omega)}\leq C\|\Delta^2u\|_{L^2(\Omega)}$ for all $u,\varphi\in H^4(\Omega)\cap H^2_0(\Omega)$ (the second inequality follows from standard elliptic regularity for the Dirichlet problem for the biharmonic operator and from the regularity assumptions of $\Omega$, see \cite[Thm. 2.20]{gazzola} for details). Therefore the bilinear form given by the left-hand side of \eqref{bilinear4} defines on $H^4(\Omega)\cap H^2_0(\Omega)$ a scalar product whose induced norm is equivalent to the standard norm of $H^4(\Omega)$. Thus, the analogue of Theorem \ref{eigenvalues_sigma} holds.
\begin{thm}
The set of eigenvalues of problem \eqref{Neumann_bad} is contained in $[0,+\infty[$. The eigenspace corresponding to the eigenvalue $\lambda=0$ has infinite dimension and all harmonic functions in $H^2(\Omega)$ are eigenfunctions associated with $\lambda=0$. Moreover, the set of positive eigenvalues consists of a sequence increasing to $+\infty$. Each positive eigenvalue has finite multiplicity and the corresponding eigenfunctions form a orthonormal basis of $\Delta(H^4(\Omega)\cap H^2_0(\Omega))$.
\end{thm}

\section{Dependence of the Neumann eigenvalues upon the Poisson's ratio}\label{sec:3}

In the first part of this section we consider the behavior of the eigenvalues of problem \eqref{Neumann_sigma} as $\sigma\rightarrow 1^-$. In the second part, we show that the positive eigenvalues of problem \eqref{Neumann_bad} and the eigenvalues of problem \eqref{Dirichlet} coincide. We start with the following theorem.
\begin{thm}\label{neumann_sigma_0}
For all $j\in\mathbb N$ it holds $\lim_{\sigma\rightarrow 1^-}\lambda_j(\sigma)=0$. Moreover, the function $\lambda_j$ from $[0,1]$ to $\mathbb R$ which maps $\sigma\in[0,1[$ to $\lambda_j(\sigma)$, and extended at $\sigma=1$ by setting $\lambda_j(1)=0$, is Lipschitz continuous on $[0,1]$.
\end{thm}
\begin{proof}
The proof is divided into three steps. In the first step we prove that $\lim_{\sigma\rightarrow 1^-}\lambda_j(\sigma)=0$ for all $j\in\mathbb N$. In the second step we prove that $\lambda_j(\sigma)$ is locally Lipschitz continuous on $[0,1[$. In the third step we prove that the function $\lambda_j(\sigma)$ extended with continuity at $\sigma=1$ is Lipschitz continuous in a neighborhood of $\sigma=1$.

{\it Step 1.} We recall that for each $\sigma\in[0,1[$ we have the following formula for $\lambda_j(\sigma)$
\begin{equation}\label{ray}
\lambda_j(\sigma)=\inf_{\substack{E\leq H^2(\Omega)\\{\rm dim}E=j}}\sup_{0\ne u\in E}\frac{\int_{\Omega}(1-\sigma)|D^2u|^2+\sigma(\Delta u)^2dx}{\int_{\Omega}u^2dx}.
\end{equation}
We also recall that the space $H^2_h(\Omega)$ is closed in $H^2(\Omega)$ and therefore it is a Hilbert space, endowed with the standard scalar product of $H^2(\Omega)$. Let $\left\{u_i\right\}_{i=1}^{\infty}$ be a set of linearly independent functions in $H^2_h(\Omega)$ such that $\int_{\Omega}u_iu_k=\delta_{ik}$ for all $i,k\in\mathbb N$. Then, from \eqref{ray} we have that for all $j\in\mathbb N$ it holds
\begin{equation*}
\lambda_j(\sigma)\leq\sup_{c_1,...,c_j\in\mathbb R}\frac{(1-\sigma)\int_{\Omega}\left|\sum_{i=1}^j c_iD^2 u_i\right|^2dx}{\int_{\Omega}\left(\sum_{i=1}^jc_iu_i\right)^2dx},
\end{equation*}
where we have chosen as $j$-dimensional space $E$ in \eqref{ray} the space generated by $\left\{u_1,...,u_j\right\}$. Then we have
\begin{multline}\label{est_conv}
\sup_{c_1,...,c_j\in\mathbb R}\frac{(1-\sigma)\int_{\Omega}\left|\sum_{i=1}^j c_iD^2 u_i\right|^2dx}{\int_{\Omega}\left(\sum_{i=1}^jc_iu_i\right)^2dx}\\ \leq\sup_{c_1,...,c_j\in\mathbb R}j(1-\sigma)\frac{\sum_{i=1}^jc_i^2\int_{\Omega}|D^2u_i|^2dx}{\sum_{i=1}^jc_i^2}\\
\leq j(1-\sigma)\max_{i=1,...,j}\int_{\Omega}|D^2u_i|^2dx,
\end{multline}
and therefore 
\begin{equation}\label{limit}
\lim_{\sigma\rightarrow 1^-}\lambda_j(\sigma)=0,
\end{equation}
for all $j\in\mathbb N$.

{\it Step 2.} For each $\sigma_1,\sigma_2\in[0,1[$ and $u\in H^2(\Omega)$ we have
\begin{multline}\label{lip_sigma}
\left|\frac{\int_{\Omega}(1-\sigma_1)|D^2u|^2+\sigma_1(\Delta u)^2dx}{\int_{\Omega}u^2dx}-\frac{\int_{\Omega}(1-\sigma_2)|D^2u|^2+\sigma_2(\Delta u)^2dx}{\int_{\Omega}u^2dx}\right|\\
\leq|\sigma_1-\sigma_2|\frac{\int_{\Omega}|D^2u|^2+(\Delta u)^2dx}{\int_{\Omega}u^2dx}
\leq (1+N)|\sigma_1-\sigma_2|\frac{\int_{\Omega}|D^2u|^2dx}{\int_{\Omega}u^2dx}\\
\leq (1+N)\frac{|\sigma_1-\sigma_2|}{1-\sigma_2}\frac{\int_{\Omega}(1-\sigma_2)|D^2u|^2+\sigma_2(\Delta u)^2dx}{\int_{\Omega}u^2dx}.
\end{multline}
From \eqref{lip_sigma} it follows that
\begin{multline}\label{lip_sigma_2}
\frac{\int_{\Omega}(1-\sigma_2)|D^2u|^2+\sigma_2(\Delta u)^2dx}{\int_{\Omega}u^2dx}\left(1-(1+N)\frac{|\sigma_1-\sigma_2|}{1-\sigma_2}\right)\\
\leq \frac{\int_{\Omega}(1-\sigma_1)|D^2u|^2+\sigma_1(\Delta u)^2dx}{\int_{\Omega}u^2dx}\\
\leq \frac{\int_{\Omega}(1-\sigma_2)|D^2u|^2+\sigma_2(\Delta u)^2dx}{\int_{\Omega}u^2dx}\left(1+(1+N)\frac{|\sigma_1-\sigma_2|}{1-\sigma_2}\right)
\end{multline}
If $\sigma_1,\sigma_2$ satisfy $(1+N)|\sigma_1-\sigma_2|<1-\sigma_2$, then taking the infimum and the supremum in \eqref{lip_sigma_2} yields
\begin{equation*}
|\lambda_j(\sigma_1)-\lambda_j(\sigma_2)|\leq(1+N)\frac{\lambda_j(\sigma_2)}{1-\sigma_2}|\sigma_1-\sigma_2|.
\end{equation*}
By repeating the same arguments above, it is possible to prove that
\begin{equation}\label{eq4}
|\lambda_j(\sigma_1)-\lambda_j(\sigma_2)|\leq(1+N)\frac{\lambda_j(\min\left\{\sigma_1,\sigma_2\right\})}{1-\min\left\{\sigma_1,\sigma_2\right\}}|\sigma_1-\sigma_2|,
\end{equation}
for all $\sigma_1,\sigma_2$ satisfying $(1+N)|\sigma_1-\sigma_2|<1-\min\left\{\sigma_1,\sigma_2\right\}$. Then the function $\lambda_j(\sigma)$ is locally Lipschitz on $[0,1[$. 

We note that from \eqref{eq4} it follows that for all $\varepsilon\in]0,1[$, the function $\lambda_j(\sigma)$ is Lipschitz continuous on $[0,1-\varepsilon]$. Moreover, from \eqref{limit} it follows that the function $\lambda_j(\sigma)$ can be extended with continuity at $\sigma=1$ by setting $\lambda_j(1):=0$. 

{\it Step 3.} Now we prove that the function $\lambda_j(\sigma)$ extended with continuity at $\sigma=1$ is Lipschitz on $[0,1]$. We note that \eqref{eq4} does not allow to prove that $\lambda_j(\sigma)$ is Lipschitz in a neighborhood of $\sigma=1$. We need a refined estimate for $|\lambda(\sigma_1)-\lambda(\sigma_2)|$ near $\sigma=1$. Let $\sigma_1,\sigma_2\in]\frac{1}{2},1[$. By using the same arguments of Step 2, we have that
\begin{multline*}
\left|\frac{\int_{\Omega}(1-\sigma_1)|D^2u|^2+\sigma_1(\Delta u)^2dx}{\int_{\Omega}u^2dx}-\frac{\int_{\Omega}(1-\sigma_2)|D^2u|^2+\sigma_2(\Delta u)^2dx}{\int_{\Omega}u^2dx}\right|\\
\leq|\sigma_1-\sigma_2|\frac{\int_{\Omega}|D^2u|^2+(\Delta u)^2dx}{\int_{\Omega}u^2dx}\leq\frac{|\sigma_1-\sigma_2|}{1-\sigma_i}\frac{\int_{\Omega}(1-\sigma_i)|D^2u|^2+\sigma_i(\Delta u)^2dx}{\int_{\Omega}u^2dx},
\end{multline*}
for $i=1,2$. Hence, from the same arguments of Step 2, we deduce that
\begin{equation*}
|\lambda_j(\sigma_1)-\lambda_j(\sigma_2)|\leq\frac{\lambda_j(\sigma_i)}{1-\sigma_i}|\sigma_1-\sigma_2|,
\end{equation*}
for all $\sigma_1,\sigma_2\in]\frac{1}{2},1[$ with $|\sigma_1-\sigma_2|<1-\sigma_i$, for $i=1,2$. In particular, we note that if $\sigma_{i_1}>\sigma_{i_2}$, then $|\sigma_1-\sigma_2|<1-\sigma_{i_2}$. Therefore 
\begin{equation}\label{estimate2}
|\lambda_j(\sigma_{1})-\lambda_j(\sigma_{2})|\leq\frac{\lambda_j(\min\left\{\sigma_1,\sigma_2\right\})}{1-\min\left\{\sigma_1,\sigma_2\right\}}|\sigma_1-\sigma_2|,
\end{equation}
for all $\sigma_1,\sigma_2\in]\frac{1}{2},1[$. Moreover, from \eqref{est_conv}, it follows that there exists a constant $C_j$ which does not depend on $\sigma$, such that
\begin{equation}\label{eq1}
\lambda_j(\sigma)\leq C_j (1-\sigma),
\end{equation}
for all $\sigma\in[0,1]$. From \eqref{estimate2} and \eqref{eq1} it follows that
\begin{equation*}
|\lambda_j(\sigma_1)-\lambda_j(\sigma_2)|\leq C_j |\sigma_1-\sigma_2|,
\end{equation*}
for all $\sigma_1,\sigma_2\in]\frac{1}{2},1]$. Then $\lambda_j(\sigma)$ is Lipschitz in a neighborhood of $\sigma=1$, hence it is Lipschitz on $[0,1]$. This concludes the proof of the theorem.
\end{proof}






Thus, the positive eigenvalues of problem \eqref{Neumann_bad} are not limiting points for the eigenvalues of problem \eqref{Neumann_sigma} as $\sigma\rightarrow 1^-$.

Now we consider problems \eqref{Neumann_bad} and \eqref{Dirichlet}. We note that, under the assumptions that $\Omega$ is of class $C^{4,\alpha}$, we have that the eigenfunctions $w$ of problem \eqref{Dirichlet} are of class $C^{4,\alpha}(\overline\Omega)$ (see \cite[Thm. 2.20]{gazzola}). We have the following theorem.
\begin{thm}\label{main}
All the positive eigenvalues of problem \eqref{Neumann_bad} coincide with the eigenvalues of problem \eqref{Dirichlet}.
\end{thm}
\begin{proof}
Let $\mu$ be an eigenvalue of problem \eqref{Dirichlet} and let $w\in H^2_0(\Omega)$ be an eigenfunction associated with $\mu$ . Let $v_0\in H^2(\Omega)\cap H^1_0(\Omega)$ be the unique solution of
\begin{equation*}
\begin{cases}
\Delta v_0=w, & {\rm in}\ \Omega,\\
v_0=0, & {\rm on}\ \partial\Omega.
\end{cases}
\end{equation*}
We set $v_h=v_0+h$ for some harmonic function $h$. Now we consider the following problem: find a harmonic function $h$ such that
\begin{equation*}
\begin{cases}
\Delta^2 v_h=\mu v_h, & {\rm in}\ \Omega,\\
\Delta v_h=0, & {\rm on}\ \partial\Omega,\\
\frac{\partial\Delta v_h}{\partial\nu}=0,&{\rm on}\ \partial\Omega.
\end{cases}
\end{equation*}
Clearly $\Delta {v_h}_{|_{\partial\Omega}}=\frac{\partial\Delta v_h}{\partial\nu}_{|_{\partial\Omega}}=0$ for all harmonic functions $h$. As for the differential equation, we have $\Delta^2(v_0+h)=\mu (v_0+h)$ if and only if $\Delta(\Delta v_0+\Delta h)=\mu (v_0+h)$, that is $\Delta w=\mu(v_0+h)$ and therefore $h=\frac{\Delta w}{\mu}-v_0$, which is clearly harmonic and belongs to $H^2(\Omega)$. Therefore each eigenvalue $\mu$ of problem \eqref{Dirichlet} is an eigenvalue of problem \eqref{Neumann_bad} and a corresponding eigenfunction is given by $v=\frac{\Delta w}{\mu}$. On the other hand, suppose that $\lambda>0$ is an eigenvalue of problem \eqref{Neumann_bad} and let $u\in \Delta(H^4(\Omega)\cap H^2_0(\Omega))$ be a corresponding eigenfunction. Then, the function $w=\Delta u$ is in $H^2_0(\Omega)$ and solves
\begin{equation*}
\begin{cases}
\Delta^2 w=\lambda w, & {\rm in}\ \Omega,\\
w=0, & {\rm on}\ \partial\Omega,\\
\frac{\partial w}{\partial\nu}=0,&{\rm on}\ \partial\Omega,
\end{cases}
\end{equation*}
therefore, $\lambda$ is an eigenvalue of problem \eqref{Dirichlet} with corresponding eigenfunction $\Delta u$.
\end{proof}

\section{Neumann and Dirichlet eigenvalues in the case of the unit ball}\label{sec:4}
In this section we consider problems \eqref{Neumann_sigma}, \eqref{Neumann_bad} and \eqref{Dirichlet} when $\Omega=B$ is the unit ball in $\mathbb R^N$ centered at zero. In this case it is possible to perform explicit computations which allow to recast the eigenvalue problems \eqref{Neumann_sigma}, \eqref{Neumann_bad}  and \eqref{Dirichlet} into suitable equations of the form $F(\lambda)=0$ and then gather informations on the behavior of the eigenvalues.

It is convenient to use the standard spherical coordinates $(r,\theta)\in \mathbb R_+\times\partial B$ in $\mathbb R^N$. We refer e.g., to \cite{mazyabook} for more details on spherical coordinates in $\mathbb R^N$. We denote by $\Delta_S$ the Laplace-Beltrami operator on the unit sphere $\partial B$ of $\mathbb R^N$. We denote by $H_l(\theta)$ a spherical harmonic of order $l\in\mathbb N_0$, where $\mathbb N_0=\mathbb N\cup\left\{0\right\}$. We recall that for all $l\in\mathbb N_0$, $H_l$ is a solution of the equation $-\Delta_S H_l=l(l+N-2)H_l$.\\
As customary, for $l\in\mathbb N_0$, we denote by $j_l$ and $i_l$ the ultraspherical and modified ultraspherical Bessel functions of the first species and order $l$ respectively, which are defined by
$$
j_l(z)=z^{1-\frac{N}{2}}J_{\frac{N}{2}-1+l}(z)\,,\ \ \ i_l(z)=z^{1-\frac{N}{2}}I_{\frac{N}{2}-1+l}(z),
$$
where $J_{\nu}(z)$ and $I_{\nu}(z)$ are the Bessel and modified Bessel functions of the first species and order $\nu$ respectively (see \cite[\S 9]{abram} for details).

We consider first problem \eqref{Dirichlet} on $B$. For the convenience of the reader we recall a result from \cite{chasman}.
\begin{lem}\label{Dirichlet_eigenfunctions}
Given an eigenvalue $\mu$ of problem \eqref{Dirichlet} on $B$, a corresponding eigenfunction $w$ is of the form $w(r,\theta)=W_l(r)H_l(\theta)$, for some $l\in\mathbb N_0$, where
\begin{equation}\label{G}
W_l(r)=\alpha j_l(\sqrt[4]{\mu}r)+\beta i_l(\sqrt[4]{\mu}r),
\end{equation}
for suitable $\alpha,\beta\in\mathbb R$.
\end{lem}
We establish now an implicit characterization of the eigenvalues of \eqref{Dirichlet} on $B$.
\begin{lem}
The eigenvalues $\mu$ of problem \eqref{Dirichlet} on $B$ are given implicitly as zeroes of the equation
\begin{equation}\label{det_dirichlet}
j_l(\sqrt[4]{\mu})i_l'(\sqrt[4]{\mu})-i_l(\sqrt[4]{\mu})j_l'(\sqrt[4]{\mu})=0.
\end{equation}
\end{lem}
\begin{proof}
By Lemma \ref{Dirichlet_eigenfunctions}, an eigenfunction $w$ associated with an eigenvalue $\mu$ is of the form $w(r,\theta)=W_l(r)H_l(\theta)$, where $W_l(r)$ is given by \eqref{G}. We recall that in spherical coordinates the Dirichlet boundary conditions are written as
\begin{equation*}
w_{|_{r=1}}=\partial_r w_{|_{r=1}}=0. 
\end{equation*}
By imposing boundary conditions to $w(r,\theta)$ we obtain a homogeneous system of two equations in two unknowns $\alpha$ and $\beta$ which has solutions if and only if its determinant vanishes. This yields formula \eqref{det_dirichlet}.
\end{proof}

Now we consider problem \eqref{Neumann_sigma} on $B$. For the convenience of the reader we recall the following result from \cite{chasman}.
\begin{lem}\label{Neumann_sigma_eigenfunctions}
Given an eigenvalue $\lambda$ of problem \eqref{Neumann_sigma} with $\sigma\in[0,1]$ on $B$, a corresponding eigenfunction $u$ is of the form $u(r,\theta)=U_l(r)H_l(\theta)$, for some $l\in\mathbb N_0$, where 
\begin{equation}\label{U}
U_l(r)=\alpha j_l(\sqrt[4]{\lambda}r)+\beta i_l(\sqrt[4]{\lambda}r),
\end{equation}
for $\alpha,\beta\in\mathbb R$.
\end{lem}
We have the following lemma on the eigenvalues of problem \eqref{Neumann_sigma} on $B$.
\begin{lem}
The eigenvalues $\lambda$ of problem \eqref{Neumann_sigma} with $\sigma\in[0,1]$ on $B$ are given implicitly as zeroes of the equation
\begin{equation}\label{det_sigma_equation}
{\rm det}M(\lambda,\sigma)=0,
\end{equation}
where $M(\lambda,\sigma)$ is the $2\times 2$ matrix defined by
\scriptsize
\begin{equation}\label{sigma_matrix}
\begin{bmatrix}
\sqrt{\lambda}j_l''(\sqrt[4]{\lambda}) +(N-1)\sqrt[4]{\lambda}\sigma j_l'(\sqrt[4]{\lambda}) & &  \sqrt{\lambda}i_l''(\sqrt[4]{\lambda})+(N-1)\sqrt[4]{\lambda}\sigma i_l'(\sqrt[4]{\lambda})\\
-l(l+N-2)\sigma j_l(\sqrt[4]{\lambda})& & -l(l+N-2)\sigma i_l(\sqrt[4]{\lambda})\\
  &  & \\
\sqrt[4]{\lambda^3}j_l'''(\sqrt[4]{\lambda})+(N-1)\sqrt{\lambda}j_l''(\sqrt[4]{\lambda})& & \sqrt[4]{\lambda^3}i_l'''(\sqrt[4]{\lambda})+(N-1)\sqrt{\lambda}i_l''(\sqrt[4]{\lambda})\\
+\sqrt[4]{\lambda}\left(1-N+l(\sigma-2)(N+l-2)\right)j_l'(\sqrt[4]{\lambda})& &+\sqrt[4]{\lambda}\left(1-N+l(\sigma-2)(N+l-2)\right)i_l'(\sqrt[4]{\lambda})\\
-l(l+N-2)(\sigma-3)j_l(\sqrt[4]{\lambda})& &-l(l+N-2)(\sigma-3)i_l(\sqrt[4]{\lambda})
\end{bmatrix}
\end{equation}
\normalsize
\end{lem}
\begin{proof}
By Lemma \ref{Neumann_sigma_eigenfunctions}, an eigenfunction $u$ associated with an eigenvalue $\lambda$ is of the form $u(r,\theta)=U_l(r)H_l(\theta)$, where $U_l(r)$ is given by \eqref{U}. We recall that in spherical coordinates the Neumann boundary conditions are written as
\begin{equation*}
\begin{cases}
(1-\sigma)\partial^2_{rr}u+\sigma\Delta u_{|_{r=1}}=0,\\
\partial_r(\Delta u)+(1-\sigma)\frac{1}{r^2}\Delta_S\left(\partial_r u-\frac{u}{r}\right)_{|_{r=1}}=0,
\end{cases}
\end{equation*}
see \cite{chasman} for details. By imposing boundary conditions to the function $u$ we obtain a system of two equations in two unknowns $\alpha$ and $\beta$, and the associated matrix is given by \eqref{sigma_matrix}. Thus the eigenvalues must solve equation \eqref{det_sigma_equation}.
\end{proof}

We give now an alternative proof of Theorem \ref{main} when $\Omega=B$ is the unit ball in $\mathbb R^N$ centered at zero based on the explicit representations of the eigenvalues discussed in this section. We have the following theorem.

\begin{thm}\label{dirichlet_vs_neumann_ball}
Equations ${\rm det}M(\lambda,1)=0$ and \eqref{det_dirichlet} admit the same non-zero solutions.
\end{thm}
\begin{proof}
We consider \eqref{sigma_matrix} with $\sigma=1$. Let $\lambda>0$ be a solution of ${\rm det}M(\lambda,1)=0$. We compute  $F(\lambda)={\rm det}M(\lambda,1)$. We have
\begin{multline}\label{F_long}
F(\lambda)=-\sqrt[4]{\lambda}l(l-1)(N+l-2)(N+l-1)\left(j_l(\sqrt[4]{\lambda})i_l'(\sqrt[4]{\lambda})-i_l(\sqrt[4]{\lambda})j_l'(\sqrt[4]{\lambda})\right)\\
+\sqrt{\lambda}l(N+1)(l+N-2)\left(j_l(\sqrt[4]{\lambda})i_l''(\sqrt[4]{\lambda})-i_l(\sqrt[4]{\lambda})j_l''(\sqrt[4]{\lambda})\right)\\
-\lambda^{3/4}(N(N-1)+l(N+l-2))\left(j_l'(\sqrt[4]{\lambda})i_l''(\sqrt[4]{\lambda})-i_l'(\sqrt[4]{\lambda})j_l''(\sqrt[4]{\lambda})\right)\\
+\lambda^{3/4}l(l+N-2)(j_l(\sqrt[4]{\lambda})i_l'''(\sqrt[4]{\lambda})-i_l(\sqrt[4]{\lambda})j_l'''(\sqrt[4]{\lambda}))\\
-\lambda(N-1)\left(j_l'(\sqrt[4]{\lambda})i_l'''(\sqrt[4]{\lambda})-i_l'(\sqrt[4]{\lambda})j_l'''(\sqrt[4]{\lambda})\right)\\
+\lambda^{5/4}\left(j_l''(\sqrt[4]{\lambda})i_l'''(\sqrt[4]{\lambda})-i_l''(\sqrt[4]{\lambda})j_l'''(\sqrt[4]{\lambda})\right).
\end{multline} 
We set $C_l^{\pm}(z)=I_{\frac{N}{2}+l}(z)J_{\frac{N}{2}-1+l}(z)\pm I_{\frac{N}{2}-1+l}(z)J_{\frac{N}{2}+l}(z)$. We use the well-known recurrence formulas for Bessel functions and their derivatives (see \cite[9.1.27 and 9.6.26]{abram}) to get
\small
\begin{equation}
j_l(z)i_l'(z)-i_l(z)j_l'(z)=z^{2-N}C_l^+(z),\label{string1}
\end{equation}
\begin{equation}
j_l(z)i_l''(z)-i_l(z)j_l''(z)=z^{1-N}\left(2zI_{\frac{N}{2}-1+l}(z)J_{\frac{N}{2}-1+l}(z)-(N-1)C_l^+(z)\right),
\end{equation}
\begin{multline}
j_l'(z)i_l''(z)-i_l'(z)j_l''(z)\\=z^{-N}\left(z^2C_l^-(z)+2lzI_{\frac{N}{2}-1+l}(z)J_{\frac{N}{2}-1+l}(z)-l(l+N-2)C_l^+(z)\right),
\end{multline}
\begin{multline}
j_l(z)i_l'''(z)-i_l(z)j_l'''(z)\\=z^{-N}\left(z^2C_l^-(z)+2(1-N+l)zI_{\frac{N}{2}-1+l}(z)J_{\frac{N}{2}-1+l}(z)\right.\\
\left.+(N(N-1)+l(l+N-2))C_l^+(z)\right),
\end{multline}
\begin{multline}
j_l'(z)i_l'''(z)-i_l'(z)j_l'''(z)\\=z^{-1-N}\left(-2z^3I_{\frac{N}{2}+l}(z)J_{\frac{N}{2}+l}(z)+(1-N+2l)z^2C_l^-(z)\right.\\
\left.+2l(1-N+l)zI_{\frac{N}{2}-1+l}(z)J_{\frac{N}{2}-1+l}(z)+l(l+N-2)(N+1)C_l^+(z)\right),
\end{multline}
\begin{multline}
j_l''(z)i_l'''(z)-i_l''(z)j_l'''(z)\\=z^{-2-N}\left(-z^4C_l^+(z)+2(N-1)z^3I_{\frac{N}{2}+l}(z)J_{\frac{N}{2}+l}(z)-(N+1)(2l+1)z^2C_l^-(z)\right.\\
\left.-2(N-3)(l-1)lzI_{\frac{N}{2}-1+l}(z)J_{\frac{N}{2}-1+l}(z)+l(l-1)(l+N-2)(l+N-1)C_l^+(z)\right).\label{stringN}
\end{multline}
\normalsize
Thanks to \eqref{string1}-\eqref{stringN}, expression \eqref{F_long} simplifies to
\begin{equation}\label{F_short}
F(\lambda)=\lambda^{5/4}\left(j_l(\sqrt[4]{\lambda})i_l'(\sqrt[4]{\lambda})-i_l(\sqrt[4]{\lambda})j_l'(\sqrt[4]{\lambda})\right).
\end{equation}
Therefore by comparing \eqref{F_short} with \eqref{det_dirichlet} we see that the non-zero eigenvalues of problem \eqref{Neumann_bad} and the eigenvalues of problem \eqref{Dirichlet} on the unit ball coincide.
\end{proof}
\begin{rem}
From Theorem \ref{neumann_sigma_0} it follows that all the eigenvalues $\lambda_j(\sigma)\rightarrow 0$ as $\sigma\rightarrow 1^-$. This means that there are infinitely many branches of solutions $\sigma\mapsto\lambda(\sigma)$ of equation \eqref{det_sigma_equation} such that $\lambda(\sigma)\rightarrow 0$ as $\sigma\rightarrow 1^-$. Theorem \ref{dirichlet_vs_neumann_ball} shows that there are also infinitely many branches $\sigma\mapsto\lambda(\sigma)$ such that $\lambda(\sigma)\rightarrow\mu$ as $\sigma\rightarrow 1^-$, for some solution $\mu>0$ to equation \eqref{det_dirichlet} (see Figure \ref{fig1}).
\end{rem}

\begin{figure}[ht!]
 \centering
    \includegraphics[width=0.6\textwidth]{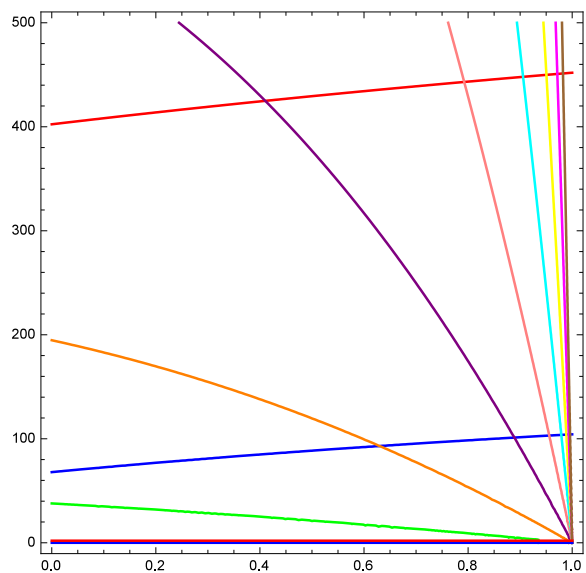}
		 \caption{ Solution branches  of equation \eqref{det_sigma_equation} with $N=2$ for $(\sigma,\lambda)\in]0,1[\times]0,500[$ . The color refers to the choice of $l$ in \eqref{det_sigma_equation}: blue ($l=0$), red ($l=1$), green ($l=2$), orange ($l=3$), purple ($l=4$), pink ($l=5$), cyan ($l=6$), yellow ($l=7$), magenta ($l=8$), brown ($l=9$).}
		\label{fig1}
\end{figure}


{\bf Acknowledgments.} The author is grateful to Professor Pier Domenico Lamberti, to Doctor Davide Buoso and to Francesco Ferraresso for useful suggestions and fruitful discussions on the argument. The author gratefully acknowledges the anonymous referee for the careful reading of the manuscript and for his useful comments. The author acknowledges financial support from the research project `Singular perturbation problems for differential operators' Progetto di Ateneo of the University of Padova. The author also acknowledges financial support from the research project `INdAM GNAMPA Project 2015 - Un approccio funzionale analitico per problemi di perturbazione singolare e di omogeneizzazione'.  The author is member of the Gruppo Nazionale per l'Analisi Matematica, la Probabilit\`a e le loro Applicazioni (GNAMPA) of the Istituto Naziona\-le di Alta Matematica (INdAM).

\bibliography{bibliography}{}
\bibliographystyle{abbrv}
\end{document}